\documentclass[12pt,final]{article}
\usepackage[margin=2.5cm,top=1.5cm]{geometry}
\usepackage[all,arc]{xy}
\usepackage{amsmath,amsthm,amssymb,enumerate}
\usepackage{color}
\newtheorem{thm}{Theorem}[section]

\newtheorem{cor}{Corollary}[section]
\newtheorem{con}{Conjecture}[section]
\newtheorem{lem}{Lemma}[section]

\theoremstyle{definition}

\newtheorem{rem}{Remark}[section]

\theoremstyle{remark}

\newcommand{\Romannum}[1]{\uppercase\expandafter{\romannumeral #1}}

\numberwithin{equation}{section}

\newcommand\keywordsname{Key words}
\newcommand\AMSname{AMS subject classifications}

\newenvironment{@abssec}[1]{%
     \if@twocolumn
       \section*{#1}%
     \else
       \vspace{.05in}\footnotesize
       \parindent .2in
         {\upshape\bfseries #1. }\ignorespaces
     \fi}
     {\if@twocolumn\else\par\vspace{.1in}\fi}

\begin{document}

\vskip6cm
\title{Notes on the sum of powers of the signless Laplacian eigenvalues of graphs
\footnote{Research supported by National Natural Science Foundation of China
(Nos. 10901061, 11071088), the Zhujiang Technology New Star Foundation of
Guangzhou (No. 2011J2200090), and Program on International Cooperation and Innovation, Department of Education, Guangdong Province (No.
2012gjhz0007).}}
\author{Lihua You \footnote{{\it{Corresponding author:\;}}ylhua@scnu.edu.cn} \qquad
Jieshan Yang  \footnote{jieshanyang1989@163.com}}
\vskip.2cm
\date{{\small
School of Mathematical Sciences, South China Normal University,\\
Guangzhou, 510631, China\\
}} \maketitle

\begin{abstract}

 \vskip.3cm
 For a graph $G$ and a non-zero real number $\alpha$, the graph invariant $S_{\alpha}(G)$ is the sum of the $\alpha^{th}$ power of the non-zero signless Laplacian eigenvalues of $G$. In this paper, we obtain the sharp bounds of $S_{\alpha}(G)$ for a connected bipartite graph $G$ on $n$ vertices and a connected graph $G$ on $n$ vertices having a connectivity less than or equal to $k$, respectively, and propose some open problems for future research.
 \vskip.2cm \noindent{\it{AMS classification:}}  05C50; 05C35; 15A18
  \vskip.2cm \noindent{\it{Keywords:}} signless Laplacian matrix; bipartite graph; connectivity.
\end{abstract}

\section{ Introduction}
\hskip.6cm
Let $G$ be a simple connected graph with vertex set $V(G)=\{v_1, v_2, \ldots, v_n\}$ and $d_i$ be the degree of the vertex $v_i$ for $i\in\{1, 2, \ldots,n\}$. Let $A(G)$ be the adjacent matrix, $D(G)$ be the diagonal matrix of vertex degrees of $G$, respectively. The Laplacian matrix of $G$ is $L(G)=D(G)-A(G)$ and the signless Laplacian matrix of $G$ is $Q(G)=D(G)+A(G)$. It is well known that both $L(G)$ and $Q(G)$ are symmetric and positive semidefinite, then we can denote the eigenvalues of $L(G)$ and $Q(G)$ by $\mu_{1}(G)\geq\mu_{2}(G)\geq\ldots\geq\mu_{n}(G)=0$ and $q_{1}(G)\geq q_{2}(G)\geq\ldots\geq q_{n}(G)\geq 0$. If no confusion, we write $\mu_{i}(G)$ as $\mu_{i}$, and $q_{i}(G)$ as $q_{i}$, respectively. The trace of the matrix $X=(x_{ij})_{n\times n}$ is defined as $tr(X)=\sum\limits_{i=1}^n x_{ii},$ which is also equal to the sum of eigenvalues of $X$. The join of $G$ and $H$, denoted by $G\vee H$, is the graph obtained by starting with a disjoint union of $G$ and $H$ and adding edges joining every vertex of $G$ to every vertex of $H.$

Let $\lambda_1, \lambda_2, \ldots, \lambda_n$ be the eigenvalues of $A(G).$ The famous graph energy $E(G),$ introduced by Gutman \cite{1978}, is defined as
$$E(G)=\sum\limits_{i=1}^{n}|\lambda_i|.$$
This quantity has a long known application in molecular-orbital theory of organic molecules and has been much investigated(see \cite{1992}, \cite{2001}).

\hskip.5cm
In \cite{1993}, Klein and Randi$\acute{c}$ defined the Kirchhoff index as $Kf(G)=\sum\limits_{i<j}r_{ij}$, where $r_{ij}$ is the effective resistance between $v_i$ and $v_j$. It was proved later by Zhu et al.\cite{1996a}, Gutman and Moher \cite{1996b} that
$$Kf(G)=n\sum\limits_{i=1}^{n-1}\frac {1}{\mu_i}.$$
The Kirchhoff index was widely used in electric circuit, probabilistic theory and chemistry(see \cite{1996b}, \cite{2004}, \cite{2008a}). Most of its results can be found in the survey \cite{2010}.

Recently, the so-called Laplacian energy $E_{L}(G)$ \cite{2006} and the Laplacian-energy-like invariant $LEL(G)$ \cite{2008b}  defined respectively as

\hskip4cm $E_L(G)=\sum\limits_{i=1}^n\mu_i^2,$ \quad $LEL(G)=\sum\limits_{i=1}^{n-1}\sqrt{\mu_i}$

\noindent have been investigated. Stevanovi$\acute{c}$ et al. \cite{2009a} showed that the LEL-variant is a well designed molecular descriptor, which has great application in chemistry. For more details on $LEL(G)$, we refer readers to the survey \cite{2011b}.

Motivated by the definition of $LEL(G)$, Jooyandeh et al.\cite{2009b} introduced the incidence energy $IE(G)$ of $G$, which is defined as $$IE(G)=\sum\limits_{i=1}^{n}\sqrt{q_i}.$$
In \cite{2009c}, relations between $IE(G)$ and $LEL(G)$  and several sharp upper bounds for $IE(G)$ are obtianed.

Since the definition of $LEL(G)$ and Kirchhoff index, zhou \cite{2008c} put forward a general form $s_{\alpha}(G)$, i.e., $$s_{\alpha}(G)=\sum\limits_{i=1}^{h}\mu_{i}^{\alpha},$$
where $\alpha$ is a non-zero real number and $h$ is the number of non-zero Laplacian eigenvalues of $G$. Zhou called it the sum of powers of Laplacian eigenvalues of $G$, and achieved some properties and bounds. More results on $s_{\alpha}$ obtained by Zhou can be found in \cite{2009d, 2009e}. In \cite{2011a}, the authors presented several bounds of $s_{\alpha}(G)$ for a connected graph $G$ in terms of its number of vertices and edges, connectivity and chromatic number respectively.

Motivated by the definition of $LEL$, $s_\alpha,$ and $IE$, Liu and Liu \cite{2012} put forward the sum of powers of the signless Laplacian eigenvalues of $G$, denoted by $$S_{\alpha}(G)=\sum\limits_{i=1}^{h}q_{i}^{\alpha},$$
where $\alpha$ is a non-zero real number and $h$ is the number of non-zero signless Laplacian eigenvalues of $G$. Obviously, $S_1(G)=2m,$ $S_{\frac{1}{2}}(G)=IE(G).$ They determined the graphs on $n$ vertices with the first, second, and third largest value of $S_\alpha$ when $\alpha>0$ and presented some bounds for $S_\alpha$ in terms of $\{n, m, Z_\alpha(G)\}$ where $m$ is the number of edges in $G$ and $Z_\alpha(G)=\sum\limits_{i=1}^{n}d_i^\alpha,$ especially in terms of  $\{n, m, Z_2(G)\}$ ($Z_2(G)$ usually written as $M_1(G),$ is called the first Zagreb index). According to the relations between $S_\alpha$ and $\{n, m, Z_2(G)\},$ some bounds for $IE$ are also presented. In \cite{2013}, Oscar Rojo and Eber Lenes derived an upper bound for $IE(G)$ of $G$ on $n$ vertices having a connectivity less than or equal to $k$, and showed that this upper bound is attained if and only if $G=K_k\vee(K_1\cup K_{n-k-1}).$ Moreover, Saieed Akbari et al. \cite{2010b} established some relations between $s_{\alpha}(G)$ and $S_{\alpha}(G)$ when $\alpha$ belongs to different intervals, that is, $S_\alpha(G)\geq s_\alpha(G)$ if $0<\alpha\leq 1$ or $2\leq\alpha\leq3,$ while $S_\alpha(G)\leq s_\alpha(G)$ if $1\leq\alpha\leq2,$ and the equality holds if and only if $G$ is a bipartite graph.

The vertex connectivity(or just connectivity) of a graph $G$, denoted by $\kappa(G)$, is the minimum number of vertices of $G$ whose deletion disconnects $G.$ It is conventional to define $\kappa(K_n)=n-1.$

Let $\mathcal{B}_n$ be the family of the connected bipartite graphs on $n$ vertices, $\mathcal{F}_n$ be the family of the simple connected graphs on $n$ vertices, respectively. Let $$\mathcal{V}_{n}^{k}=\{G\in\mathcal{F}_n| \kappa(G)\leq k\}.$$

In this paper, we will derive a sharp bound of $S_{\alpha}(G)$ in $\mathcal{B}_n$ in section 3, and derive a sharp bound of $S_{\alpha}(G)$ in $\mathcal{V}_{n}^{k}$ in section 4, respectively, and propose some open problems in these sections for future research.

\section{Preliminaries}
\hskip.6cm  In this section, we introduce some basic properties which we need to use in the proofs of our main results.

\begin{lem}\label{lem21}{\rm (\cite{2007})}
Let $G$ be a graph with $n$ vertices and $e$ be an edge of $G$. Then $$0\leq q_n(G-e)\leq q_n(G)\leq q_{n-1}(G-e)\leq q_{n-1}(G)\leq\cdots\leq q_1(G-e)\leq q_1(G).$$
\end{lem}

Note that $\sum\limits_{i=1}^{n}q_i(G)-\sum\limits_{i=1}^{n}q_i(G-e)=2.$ By Lemma \ref{lem21}, it immediately follows

\begin{thm}\label{thm21}
Let $e$ be an edge of $G$. Then $S_{\alpha}(G)> S_{\alpha}(G-e)$ for $\alpha>0$, and $S_{\alpha}(G)< S_{\alpha}(G-e)$ for $\alpha<0$.
\end{thm}

\begin{lem}\label{lem22}{\rm (\cite{2007b})}
If $G$ is bipartite, then $Q(G)$ and $L(G)$ share the same eigenvalues.
\end{lem}

\section{Bounding $S_{\alpha}(G)$ in $\mathcal{B}_n$}
\hskip.6cm In this section, we derive a sharp bound of $S_{\alpha}(G)$ with $\alpha\leq1$ for a connected bipartite graph $G$ on $n$ vertices, and  propose an open problem with the bound of $\alpha>1$.

we can  see that $S_\alpha(G)=s_\alpha(G)$ for a bipartite graph $G$ by Lemma \ref{lem22}, so the following results in this section also hold for $s_\alpha.$

Let $\sigma(M)$ be the spectrum of the matrix $M$. By simple calculation, $\sigma(Q(K_{r,s}))=\{r+s, r^{[s-1]}, s^{[r-1]}, 0\}$ where $\lambda^{[t]}$ means that $\lambda$ is an eigenvalue with multiplicity $t$.

From Theorem \ref{thm21}, we immediately have

\begin{thm}\label{thm31}
Let $G$ be a bipartite graph with $r$ and $s$ vertices in its two partite sets. Then we have

{\rm (1)} If $\alpha>0$, then $S_{\alpha}(G)\leq(r+s)^\alpha+(r-1)s^\alpha+(s-1)r^\alpha$
with equality if and only if $G=K_{r,s}$;

{\rm (2)} If $\alpha<0$, then $S_{\alpha}(G)\geq(r+s)^\alpha+(r-1)s^\alpha+(s-1)r^\alpha$
with equality if and only if $G=K_{r,s}$.
\end{thm}

\begin{thm}\label{thm32}
Let $G$ be a bipartite graph with $n$ vertices and $\alpha\leq1$. Then we have

{\rm (1)} If $\alpha<0$, then $S_{\alpha}(G)\geq n^\alpha+(\left\lfloor\frac{n}{2}\right\rfloor-1)\left\lceil\frac{n}{2}\right\rceil^\alpha+(\left\lceil\frac{n}{2}\right\rceil-1)\left\lfloor\frac{n}{2}\right\rfloor^\alpha$
with equality if and only if $G=K_{\left\lfloor\frac{n}{2}\right\rfloor,\left\lceil\frac{n}{2}\right\rceil}$;

{\rm (2)} If $0<\alpha\leq1$, then $S_{\alpha}(G)\leq n^\alpha+(\left\lfloor\frac{n}{2}\right\rfloor-1)\left\lceil\frac{n}{2}\right\rceil^\alpha+(\left\lceil\frac{n}{2}\right\rceil-1)\left\lfloor\frac{n}{2}\right\rfloor^\alpha$
with equality if and only if $G=K_{\left\lfloor\frac{n}{2}\right\rfloor,\left\lceil\frac{n}{2}\right\rceil}$.
\end{thm}

\begin{proof}
The proof of (2) is similar to (1). Now we show (1) holds.

If $\alpha<0$, let $G_*$ be a bipartite graph with $r$ and $s$ vertices in its two partite sets, having the minimum value of $S_\alpha$ among all the connected bipartite graphs with $n$ vertices. Without loss of generality, assume that $1\leq r\leq s$. By Theorem \ref{thm31}, $G_*=K_{r,s}$ for some $r\in \{1, 2, \ldots, \left\lfloor\frac{n}{2}\right\rfloor\}$ with $r+s=n.$ Note that $\sigma(Q(K_{r,s}))=\{r+s, r^{[s-1]}, s^{[r-1]}, 0\}$. Then

\vskip 0.1cm
\hskip2cm $S_\alpha(K_{r,s})=n^\alpha+(r-1)s^\alpha+(s-1)r^\alpha$

\hskip3.6cm $=n^\alpha+(r-1)(n-r)^\alpha+(n-r-1)r^\alpha$

\hskip3.6cm $=n^\alpha-[(n-r)^{\alpha+1}+r^{\alpha+1}]+(n-1)[(n-r)^\alpha+r^\alpha].$

Let $f(r)=-[(n-r)^{\alpha+1}+r^{\alpha+1}]+(n-1)[(n-r)^\alpha+r^\alpha]$ with $1\leq r\leq \left\lfloor\frac{n}{2}\right\rfloor.$ Then

\vskip 0.1cm

\hskip2cm  $f'(r)=(\alpha+1)[(n-r)^\alpha-r^\alpha]-\alpha(n-1)[(n-r)^{\alpha-1}-r^{\alpha-1}].$

\vskip 0.1cm 

 If $r=\left\lfloor\frac{n}{2}\right\rfloor=\frac{n}{2},$ then $n-r=r$, and therefore $f'(r)=0.$

 Otherwise, $r<\frac{n}{2},$ i.e., $n-r>r.$ By Cauchy mean-value Theorem, there exists $\xi\in (r, n-r)$ satisfying
$$\frac{(n-r)^{\alpha-1}-r^{\alpha-1}}{(n-r)^\alpha-r^\alpha}=\frac{(\alpha-1)\xi^{\alpha-2}}{\alpha\xi^{\alpha-1}}=\frac{\alpha-1}{\alpha\xi}.$$

Thus we have

\hskip2.5cm  $f'(r)=[(n-r)^\alpha-r^\alpha][(\alpha+1)-\alpha(n-1)\cdot\frac{(n-r)^{\alpha-1}-r^{\alpha-1}}{(n-r)^\alpha-r^\alpha}]$

\hskip3.5cm $=[(n-r)^\alpha-r^\alpha][(\alpha+1)-(\alpha-1)\cdot\frac{n-1}{\xi}].$

\vskip 0.1cm
Noting that $\alpha<0$,  $n-r>r$ and $0<r<\xi<n-r\leq n-1,$ we have $\alpha-1<0,$ $(n-r)^\alpha-r^\alpha<0,$ $\frac{n-1}{\xi}>1$ and $(\alpha+1)-(\alpha-1)\frac{n-1}{\xi}>(\alpha+1)-(\alpha-1)=2>0.$ Hence, $f'(r)<0,$ that is, $f(r)$ is decreasing for $1\leq r\leq\left\lfloor\frac{n}{2}\right\rfloor.$

Therefore, $S_\alpha(K_{r,n-r})=n^\alpha+f(r)$ with $1\leq r\leq\left\lfloor\frac{n}{2}\right\rfloor$ is minimum if and only if $r=\left\lfloor\frac{n}{2}\right\rfloor.$ It follows that $G_*=K_{\left\lfloor\frac{n}{2}\right\rfloor,\left\lceil\frac{n}{2}\right\rceil},$ and
$S_{\alpha}(G)\geq n^\alpha+(\left\lfloor\frac{n}{2}\right\rfloor-1)\left\lceil\frac{n}{2}\right\rceil^\alpha+(\left\lceil\frac{n}{2}\right\rceil-1)\left\lfloor\frac{n}{2}\right\rfloor^\alpha.$
\end{proof}

\begin{rem}\label{rem31}
 Since $Q(G)$ and $L(G)$ share the same eigenvalues if $G$ is bipartite by Lemma \ref{lem22}, $LEL(G)=S_{\frac{1}{2}}(G)$ and $Kf(G)=nS_{-1}(G)$ for $G$ is a bipartite graph. Hence,  Theorem \ref{thm31} and Theorem \ref{thm32} generalize the results of Liu and Huang for the Laplacian-energy-like invariant(Corollary 2.4, \cite{2011b}) and the results of Yang for the Kirchhoff index(Theorem 3.1, \cite{2012b}). In our proof, some techniques in \cite{2011a} are referred.
\end{rem}

Naturally, we put forward the following conjecture to close this section.

\begin{con}\label{con31}
Let $G$ be a bipartite graph with $n$ vertices. If $\alpha>1,$ then  $$S_{\alpha}(G)\leq n^\alpha+(\left\lfloor\frac{n}{2}\right\rfloor-1)\left\lceil\frac{n}{2}\right\rceil^\alpha+(\left\lceil\frac{n}{2}\right\rceil-1)\left\lfloor\frac{n}{2}\right\rfloor^\alpha,$$
with equality if and only if $G=K_{\left\lfloor\frac{n}{2}\right\rfloor,\left\lceil\frac{n}{2}\right\rceil}$.
\end{con}

\section{Bounding $S_\alpha(G)$ in $\mathcal{V}_n^k$}
\hskip.6cm
In this section, we characterize the extremal graph of $S_\alpha(G)$ in $\mathcal{V}_n^k$ and derive a sharp upper bound of $S_\alpha(G)$ with $\alpha\geq1$ in $\mathcal{V}_n^k.$ Moreover, we propose an open problem with the bound of $\alpha<1$

By simple calculation, $\sigma(Q(K_n))=\{2n-2, (n-2)^{[n-1]}\}.$

Actually, Theorem \ref{thm21} implies that

\begin{thm}\label{thm41}
Let $G\in\mathcal{F}_n$. Then we have\\
{\rm (1)} If $\alpha>0,$ then
$S_\alpha(G)\leq 2^\alpha(n-1)^\alpha+(n-1)(n-2)^\alpha$
with equality if and only if $G=K_n.$
{\rm (2)} If $\alpha<0,$ then $S_\alpha(G)\geq 2^\alpha(n-1)^\alpha+(n-1)(n-2)^\alpha$
with equality if and only if $G=K_n.$
\end{thm}

Throughout the following paper, let $G^*$, $G_*$ be the graphs having the maximum and the minimum value of $S_\alpha(G)$ among the graphs in $\mathcal{V}_n^k$, respectively.
Let $|U|$ be the cardinality of a finite set $U$, and $G(i)=K_k\vee(K_i\cup K_{n-k-i})$ where $i\in\{1, 2, \ldots, \lfloor\frac{n-k}{2}\rfloor\}.$

\begin{thm}\label{thm42}
Let $n,$ $k$ be positive integers with $1\leq k\leq n-1,$ $G^*(G_*)$ be defined as above. Then\\
{\rm (1)} $G^*\in\{G(1), G(2), \cdots, G(\lfloor\frac{n-k}{2}\rfloor)\}$ when $\alpha>0;$\\
{\rm (2)} $G_*\in\{G(1), G(2), \cdots, G(\lfloor\frac{n-k}{2}\rfloor)\}$ when $\alpha<0.$
\end{thm}

\begin{proof}
The proof of (2) is similar to (1). We omit it. Now we show (1) holds.

Let $G$ be any graph in $\mathcal{V}_n^k$ and $\alpha>0$.

{\bf Case 1: } $k=n-1.$

From Theorem \ref{thm41}, $S_\alpha(G)\leq S_\alpha(K_n)$ with equality if and only if $G=K_n.$ Note that $K_n=G(1),$ the result is true for $k=n-1.$

{\bf Case 2: } $1\leq k\leq n-2.$

By Theorem \ref{thm21}, there is $G^*$ in $\mathcal{V}_n^k.$ Let $U\subseteq V(G^*)$ such that $G^*-U$ is a disconnected graph and $|U|=\kappa(G^*).$ Hence, $|U|\leq k.$ Let $G_1, G_2, \ldots, G_r$ be the connected components of $G^*-U.$

First, we claim that $r=2.$ If $r>2,$ then we can construct a graph $H=G^*+e$ where $e$ is an edge connecting a vertex in $G_1$ with a vertex in $G_2.$ Clearly, $\kappa(H)\leq |U|\leq k$ since $H$ is connected and $H-U=G^*+e-U$ is disconnected. Thus, $H\in \mathcal{V}_n^k$ and $G^*=H-e.$ By Theorem \ref{thm21}, $S_\alpha(G^*)<S_\alpha(H),$ which is a contradiction. Therefore $r=2,$ that is, $G^*-U=G_1\cup G_2.$

Second, we claim that $\kappa(G^*)=k$. If $\kappa(G^*)<k,$ then $|U|<k.$ Construct a graph $H=G^*+e$ where $e$ is an edge joining a vertex $u\in V(G_1)$ with a vertex $v\in V(G_2).$ Hence, $\kappa(H)\leq |U|+1\leq k$ since $H-U$ is a connected graph and $H-U\cup\{u\}$ is disconnected. Therefore $H\in\mathcal{V}_n^k.$ By Theorem \ref{thm21}, $S_\alpha(G^*)<S_\alpha(H),$ which is also a contradiction. Thus, $\kappa(G^*)=k.$

Let $|V(G_1)|=i$. Then $|V(G_2)|=n-k-i.$ Repeating application of Theorem \ref{thm21} enables to write $G^*=K_k\cup(K_i\cup K_{n-k-i})=G(i)$ where $i\in\{1, 2, \ldots, \lfloor\frac{n-k}{2}\rfloor\}.$
\end{proof}

\begin{rem}\label{rem41}
In our proofs of Theorem \ref{thm42}, some techniques in \cite{2013} are referred.
\end{rem}

When $\alpha\geq1,$ we search for the value of $i$ for which $S_\alpha(G(i))$ ($i\in\{1, 2, \ldots, \lfloor\frac{n-k}{2}\rfloor\}$) is maximum. In this proof, we need the spectrum of $Q(G(i)),$ which is given as follows in \cite{2013}.

\begin{lem}\label{lem41}{\rm (\cite{2013})}
The spectrum of $Q(G(i))$ is
$$\sigma(Q(G(i)))=\{q_1(i), q_2, q_3(i), (n-2)^{[k-1]}, (k+i-2)^{[i-1]}, (n-i-2)^{[n-k-i-1]}\},$$
where

\hskip2cm $q_1(i)=n-2+\frac{k}{2}+\frac{1}{2}\sqrt{(k-2n)^2+16i(k-n+i)},$ \quad $q_2=n-2,$

\noindent and

\hskip2cm $q_3(i)=n-2+\frac{k}{2}-\frac{1}{2}\sqrt{(k-2n)^2+16i(k-n+i)}.$

\end{lem}

\begin{thm}\label{thm43}
Let $n,$ $k$ be positive integers with $1\leq k\leq n-1,$ $G\in\mathcal{V}_n^k$ and $\alpha\geq1.$ Then
$$S_\alpha(G)\leq b_\alpha(n,k) \eqno (4.1)$$
where

\hskip2cm $b_\alpha(n,k)=k(n-2)^\alpha+(n-k-2)(n-3)^\alpha$

\hskip3.75cm       $+\left[n-2+\frac{k}{2}+\frac{1}{2}\sqrt{(k-2n)^2+16(k-n+1)}\right]^\alpha$

\hskip3.75cm       $+\left[n-2+\frac{k}{2}-\frac{1}{2}\sqrt{(k-2n)^2+16(k-n+1)}\right]^\alpha.$

\vskip 0.1cm
\noindent The equality $(4.1)$ holds if and only if $G=K_k\vee(K_1\cup K_{n-k-1}).$
\end{thm}

\begin{proof}
Let $G^*$ be defined as above.
 Then $G^*=G(i)$ for some $i\in\{1, 2, \ldots, \lfloor\frac{n-k}{2}\rfloor\}$ by  $\alpha\geq1$ and Theorem \ref{thm42}. By Lemma \ref{lem41}, we have

\hskip1cm $S_\alpha(G(i))=k(n-2)^\alpha+(i-1)(k+i-2)^\alpha+(n-k-i-1)(n-i-2)^\alpha$

\hskip3.05cm   $+\left[n-2+\frac{k}{2}+\frac{1}{2}\sqrt{(k-2n)^2+16i(k-n+i)}\right]^\alpha$

\hskip3.05cm  $+\left[n-2+\frac{k}{2}-\frac{1}{2}\sqrt{(k-2n)^2+16i(k-n+i)}\right]^\alpha$

\hskip2.7cm $=k(n-2)^\alpha+(k+i-2)^{\alpha+1}+(n-i-2)^{\alpha+1}$

\hskip3.05cm $-(k-1)[(k+i-2)^\alpha+(n-i-2)^\alpha]$

\hskip3.05cm $+\left[n-2+\frac{k}{2}+\frac{1}{2}\sqrt{(k-2n)^2+16i(k-n+i)}\right]^\alpha$

\hskip3.05cm $+\left[n-2+\frac{k}{2}-\frac{1}{2}\sqrt{(k-2n)^2+16i(k-n+i)}\right]^\alpha$

\noindent Let

\hskip1cm $f(x)=(x+k-2)^{\alpha+1}+(n-2-x)^{\alpha+1}-(k-1)\left[(x+k-2)^\alpha+(n-2-x)^\alpha\right]$

\hskip2.25cm $+\left[n-2+\frac{k}{2}+\frac{1}{2}\sqrt{(k-2n)^2+16x(k-n+x)}\right]^\alpha$

\hskip2.25cm $+\left[n-2+\frac{k}{2}-\frac{1}{2}\sqrt{(k-2n)^2+16x(k-n+x)}\right]^\alpha$

\noindent with $1\leq x\leq\lfloor\frac{n-k}{2}\rfloor.$ Then
\vskip 0.1cm


$f'(x)=\alpha(k-1)[(n-2-x)^{\alpha-1}-(x+k-2)^{\alpha-1}]-(\alpha+1)[(n-2-x)^\alpha-(x+k-2)^\alpha]$

\hskip1.5cm $+\frac{4\alpha[2x-(n-k)]}{\sqrt{(k-2n)^2+16x(k-n+x)}}\cdot\{\left[n-2+\frac{k}{2}+\frac{1}{2}\sqrt{(k-2n)^2+16x(k-n+x)}\right]^{\alpha-1}$

\hskip1.5cm $-\left[n-2+\frac{k}{2}-\frac{1}{2}\sqrt{(k-2n)^2+16x(k-n+x)}\right]^{\alpha-1}\}.$

\vskip 0.1cm

 If $x=\lfloor\frac{n-k}{2}\rfloor=\frac{n-k}{2},$ then $n-x=x+k,$ so $f'(x)=0.$

 Otherwise, $x<\frac{n-k}{2},$ i.e., $n-x>x+k$ and therefore $n-x-2>x+k-2.$ By Cauchy mean-value Theorem, there exits $\xi\in(x+k-2, n-x-2)$ satisfying
$$\frac{(n-2-x)^{\alpha-1}-(x+k-2)^{\alpha-1}}{(n-2-x)^\alpha-(x+k-2)^\alpha}=\frac{(\alpha-1)\xi^{\alpha-2}}{\alpha\xi^{\alpha-1}}=\frac{\alpha-1}{\alpha\xi}.$$

 Thus we have

$f'(x)=[(n-2-x)^\alpha-(x+k-2)^\alpha]\cdot\left[\alpha(k-1)\cdot\frac{(n-2-x)^{\alpha-1}-(x+k-2)^{\alpha-1}}{(n-2-x)^\alpha-(x+k-2)^\alpha}-(\alpha+1)\right]$

\hskip1.5cm $+\frac{4\alpha[2x-(n-k)]}{\sqrt{(k-2n)^2+16x(k-n+x)}}\cdot\{\left[n-2+\frac{k}{2}+\frac{1}{2}\sqrt{(k-2n)^2+16x(k-n+x)}\right]^{\alpha-1}$

\hskip1.5cm $-\left[n-2+\frac{k}{2}-\frac{1}{2}\sqrt{(k-2n)^2+16x(k-n+x)}\right]^{\alpha-1}\}.$

\hskip1.05cm $=[(n-2-x)^\alpha-(x+k-2)^\alpha]\cdot[(\alpha-1)\cdot\frac{k-1}{\xi}-(\alpha+1)]$

\hskip1.5cm $+\frac{4\alpha[2x-(n-k)]}{\sqrt{(k-2n)^2+16x(k-n+x)}}\cdot\{\left[(n-2+\frac{k}{2}+\frac{1}{2}\sqrt{(k-2n)^2+16x(k-n+x)}\right]^{\alpha-1}$

\hskip1.5cm $-\left[n-2+\frac{k}{2}-\frac{1}{2}\sqrt{(k-2n)^2+16x(k-n+x)}\right]^{\alpha-1}\}.$

\vskip 0.1cm
Noting that  $\alpha\geq1,$  $x<\frac{n-k}{2},$ $n-x-2>x+k-2$ and $\xi>x+k-2\geq k-1,$ we have $f'(x)<0,$ that is, $f(x)$ is decreasing for $1\leq x\leq\lfloor\frac{n-k}{2}\rfloor.$

Therefore, $S_\alpha(G(i))=k(n-2)^\alpha+f(i)$ with $1\leq i\leq\lfloor\frac{n-k}{2}\rfloor$ is maximum if and only if $i=1.$ It followings that $G^*=K_k\vee(K_1\cup K_{n-k-1}).$
\end{proof}

Note that $S_1(G)=2m,$ where $m$ is the number of edges in $G.$ From Theorem \ref{thm43}, we have
\begin{cor}\label{cor41}
Let $n,$ $k$ be positive integers with $1\leq k\leq n-1,$ and $G$ be any graph in $\mathcal{V}_n^k$ with $m$ edges. Then
 $m\leq \frac{1}{2}b_1(n,k)=\frac{1}{2}(n^2-3n+2k+2),$
 with equality if and only if $G=K_k\vee(K_1\cup K_{n-k-1}).$
\end{cor}

 Note that $E_L(G)=tr(L(G)^2)=tr[(D(G)-A(G))^2],$ and $S_2(G)=tr(Q(G)^2)=tr[(D(G)+A(G))^2].$ Since $tr[D(G)A(G)]=0,$ $tr[(D(G)+A(G))^2]=tr[(D(G)-A(G))^2],$ which implies that $E_L(G)=S_2(G).$ So we can obtain the bound of $E_L(G)$ in $\mathcal{V}_n^k$ as follows.

\begin{cor}\label{cor42}
Let $n,$ $k$ be positive integers with $1\leq k\leq n-1,$ and $G\in \mathcal{V}_n^k.$ Then

\hskip2cm $E_L(G)\leq b_2(n,k)=n^3+2n^2+(2k+5)n+k^2-k-2,$
\vskip0.1cm
\noindent with equality if and only if $G=K_k\vee(K_1\cup K_{n-k-1}).$
\end{cor}

At this point, we recall that the edge connectivity of $G,$ denoted by $\varepsilon(G),$ is the minimum number of edges whose deletion disconnects $G.$
Let $\varepsilon_n^k=\{G\in\mathcal{F}_n|\varepsilon(G)\leq k\}.$

\begin{cor}\label{cor43}
Let $n,$ $k$ be positive integers with $1\leq k\leq n-1,$ $G$ be any graph in $\varepsilon_n^k$ and $\alpha\geq1.$ Then
$S_\alpha(G)\leq b_\alpha(n,k),$
 with equality  if and only if $G=K_k\vee(K_1\cup K_{n-k-1}).$
\end{cor}

\begin{proof}
Since $\kappa(G)\leq\varepsilon(G),$ it follows $\varepsilon_n^k\subseteq\mathcal{V}_n^k.$ Let $G\in\varepsilon_n^k,$ the corollary follows from the fact $K_k\vee(K_1\cup K_{n-k-1})\in\varepsilon_n^k$.
\end{proof}

In \cite{2013}, the authors proved that $IE(G)\leq b_{\frac{1}{2}}(n,k)$ for any graph $G$ in $\mathcal{V}_n^k,$ and the equality holds if and only if $G=K_k\vee(K_1\cup K_{n-k-1}).$ Note that $IE(G)=S_{\frac{1}{2}}(G).$ From the above facts, Theorem \ref{thm21} and Theorem \ref{thm42}, we obtain the following conjecture.

\begin{con}\label{con44}
Let $n,$ $k$ be positive integers with $1\leq k\leq n-1,$ $G\in \mathcal{V}_n^k$ and $\alpha<1.$ Then we have

{\rm (1)} If $0<\alpha<1,$ then
$S_\alpha(G)\leq b_\alpha(n,k)$ with equality if and only if $G=K_k\vee(K_1\cup K_{n-k-1}).$





{\rm (2)} If $\alpha<0,$ then $S_\alpha(G)\geq b_\alpha(n,k)$ with equality if and only if $G=K_k\vee(K_1\cup K_{n-k-1}).$
\end{con}


\end{document}